\documentclass[11pt, reqno]{amsart}
\usepackage{indentfirst, amssymb, amsmath, amsthm, mathrsfs, setspace, indentfirst, enumerate,  mathrsfs, amsmath, amsthm}
\usepackage[bookmarksnumbered, colorlinks, plainpages]{hyperref}
\usepackage{mathrsfs}
\usepackage{cite}
\usepackage{graphicx}
\usepackage{float}
\newtheorem*{theo4.A}{Theorem 4.A}
\newtheorem*{theo4.B}{Theorem 4.B}

\newtheorem*{theoA}{Theorem A}
\newtheorem*{theoB}{Theorem B}
\newtheorem*{theoC}{Theorem C}
\newtheorem*{theoD}{Theorem D}

\newtheorem*{theo2.E}{Theorem 2.E}
\newtheorem*{theo2.F}{Theorem 2.F}

\newtheorem*{theo3.A}{Theorem 3.A}
\newtheorem*{theo3.B}{Theorem 3.B}
\newtheorem*{ques3.A}{Question 3.A}

\newtheorem*{cor A}{Corollary A}
\newtheorem*{cor B}{Corollary B}
\newtheorem*{que1}{Question 1.1}
\newtheorem*{que2}{Question 1.2}

\newtheorem{theo}{Theorem}[section]
\newtheorem{lem}{Lemma}[section]

\newtheorem{defi}{Definition}[section]
\newtheorem{rem}{Remark}[section]

\newcommand{\ol}{\overline}
\newcommand{\be}{\begin{equation}}
\newcommand{\ee}{\end{equation}}
\newcommand{\beas}{\begin{eqnarray*}}
\newcommand{\eeas}{\end{eqnarray*}}
\newcommand{\bea}{\begin{eqnarray}}
\newcommand{\eea}{\end{eqnarray}}

\numberwithin{equation}{section}
\begin{document}
\title[I\MakeLowercase{mproved} B\MakeLowercase{ohr Inequalities Involving} F\MakeLowercase{r\'{e}chet Derivatives Associated with} H\MakeLowercase{olomorphic Mappings in} B\MakeLowercase{anach Spaces}
]{ I\MakeLowercase{mproved} B\MakeLowercase{ohr Inequalities Involving} F\MakeLowercase{r\'{e}chet Derivatives Associated with} H\MakeLowercase{olomorphic Mappings in} B\MakeLowercase{anach Spaces}}

\date{}
\author[N. S\MakeLowercase{arkar and} P. D\MakeLowercase{as}]{N\MakeLowercase{abadwip} S\MakeLowercase{arkar and} P\MakeLowercase{radip} D\MakeLowercase{as}}

\address{Amity School of Applied Sciences, Amity University Mumbai, Panvel, Navi Mumbai, Maharashtra-410206, India.}
\email{nsarkar@mum.amity.edu, nabadwipsarkar52@gmail.com}
\address{Department of Mathematics, Raiganj University, Raiganj, West Bengal-733134, India.}
\email{pradipsmath@gmail.com}

\renewcommand{\thefootnote}{}
\footnote{*\emph{Corresponding Author}: Pradip Das}
\renewcommand{\thefootnote}{\arabic{footnote}}
\setcounter{footnote}{0}

\begin{abstract}
Motivated by recent advances in derivative Bohr inequalities and their refinements, we investigate the Bohr phenomenon for holomorphic mappings in complex Banach spaces associated with Schwarz functions. We establish sharp Bohr-type inequalities involving higher-order Fr\'{e}chet derivatives for both $|F(z)|$ and its refined counterpart $|F(z)|^2$. The corresponding Bohr radii are determined and shown to be best possible. Our results provide affirmative answers to Questions 1.1 and 1.2 and extend several classical and recent Bohr inequalities from the unit disk to the setting of holomorphic mappings on Banach spaces.
\end{abstract}

\thanks{Typeset by \AmS -\LaTeX}
\maketitle
\section{{\bf Introduction}}

Let $H^\infty$ denote the Banach space of all bounded analytic functions on the unit disk
$\mathbb{U}=\{z\in\mathbb{C}:|z|<1\},$
equipped with the norm $\|f\|_\infty=\sup_{z\in\mathbb{D}}|f(z)|.$

A classical theorem of Bohr \cite{Bohr1914} states that if $f(z)=\sum_{n=0}^{\infty}a_n z^n \in H^\infty,$
then
\[
B_0(f,r):=|a_0|+\sum_{n=1}^{\infty}|a_n|r^n
\leq \|f\|_\infty
\]
for all $r\leq 1/6$. Shortly thereafter, Riesz, Schur, and Wiener independently improved the radius to $1/3$, proving that
\[
|a_0|+\sum_{n=1}^{\infty}|a_n|r^n \leq \|f\|_\infty
\]
whenever $r\leq 1/3$, and that the constant $1/3$ is sharp. This result is now known as the classical Bohr theorem, while the number $1/3$ is referred to as the \emph{Bohr radius}. The sharpness follows from the family of disk automorphisms
\[
\varphi_a(z)=\frac{a-z}{1-az}, \qquad 0\le a<1.
\]

Although Bohr's original paper already contained the essential ideas leading to the sharp radius, a variety of alternative proofs and extensions have subsequently appeared. Comprehensive surveys of the classical theory and its developments can be found in \cite{AAP2017,GMR2018}. It is worth noting that no extremal function in $H^\infty$ attains the Bohr radius exactly; rather, the radius $1/3$ arises as a limiting value of suitable extremal families (see \cite{AKP2019,GMR2018,KP2017}).

During the past two decades, Bohr's phenomenon has attracted considerable attention and has been investigated in a wide range of settings, including multidimensional complex analysis, harmonic and quasi-conformal mappings, operator-valued functions, Dirichlet series, and various subclasses of analytic functions. Numerous refinements and generalizations have been established; see, for example, \cite{AA2013,BD2018,BB2004,KS2022} and the references therein.


In order to present our results in a clear and systematic manner, we first fix some
notation. For a point $z=(z_1,z_2,\ldots,z_n)\in\mathbb{C}^n$, we denote by $\|z\|_\infty:=\max_{1\le i\le n}|z_i|$
the supremum norm. The unit polydisc in $\mathbb{C}^n$ is then defined by
$\mathbb{U}^n:=\{z\in\mathbb{C}^n:\ \|z\|_\infty<1\}.$\par
Let $X$ and $Y$ be complex Banach spaces.
\begin{defi}
Let $k\in\mathbb{N}$. A mapping $P\colon X\to Y$ is called a \emph{homogeneous polynomial of degree $k$}
if there exists a $k$-linear mapping $u\colon X^k\to Y$ such that
\[
P(x)=u(x,\ldots,x), \qquad x\in X.
\]
\end{defi}

Note that if $P$ is a homogeneous polynomial of degree $k$, then
\[
P(\lambda x)=\lambda^k P(x), \qquad x\in X,\ \lambda\in\mathbb{C}.
\]
Throughout this paper, the degree of a homogeneous polynomial is indicated by a subscript.
That is, if $P_k$ is a homogeneous polynomial, then its degree is $k$.
Moreover, if $P_k$ is a $k$-homogeneous polynomial from $X$ into $Y$, then there exists a unique
symmetric $k$-linear mapping $u$ such that
\[
P_k(x)=u(x,\ldots,x), \qquad x\in X.
\]

Let $D\subset X$ be a domain and let $F\colon D\to Y$ be a holomorphic mapping.
For $z\in D$, denote by $D^kF(z)$ the $k$-th Fr\'{e}chet derivative of $F$ at $z$.
If $D$ contains the origin, then $F$ admits the expansion
\begin{equation}
F(z)=\sum_{k=0}^{\infty}\frac{1}{k!}\,D^kF(0)(z^k)
\label{eq:Taylor}
\end{equation}
in a neighbourhood of the origin.
Since $\frac{1}{k!}D^kF(0)(z^k)$ is a homogeneous polynomial of degree $k$, we shall use the notation
\[
P_k(z):=\frac{1}{k!}D^kF(0)(z^k)
\]
throughout this paper.
If $D$ is a bounded balanced domain in $X$ and $F(D)$ is bounded, then the series
\eqref{eq:Taylor} converges uniformly on $rD$ for each $r\in(0,1)$.

Let $F\colon D\to Y$ be holomorphic. For $k\in\mathbb{N}$, we say that $z=0$ is a zero of order $k$ of $F$
if
\[
F(0)=0,\quad DF(0)=0,\ \ldots,\ D^{k-1}F(0)=0,\quad \text{but } D^kF(0)\neq 0.
\]

Let $B_X$ and $B_Y$ denote the open unit balls of the Banach spaces $X$ and $Y$, respectively.
A holomorphic mapping $\mu\colon B_X\to B_Y$ with $\mu(0)=0$ is called a \emph{Schwarz mapping}.
If $\mu_k$ is a Schwarz mapping such that $z=0$ is a zero of order $k$ of $\mu$, then the following estimate holds
(see, for example, \cite[Lemma~6.1.28]{GK2003}):
\begin{equation}
\|\mu_k(z)\|_Y \le \|z\|_X^{\,k}, \qquad z\in B_X.
\label{eq:Schwarz}
\end{equation}

Let $L(X,\mathbb{C})$ denote the space of continuous linear operators from $X$ into $\mathbb{C}$.
For each $x\in X\setminus\{0\}$, define
\[
T(x):=\{\,\ell_x\in L(X,\mathbb{C}) : \ell_x(x)=\|x\|,\ \|\ell_x\|=1\,\}.
\]
By the Hahn--Banach theorem, the set $T(x)$ is nonempty.

The classical Bohr phenomenon has been extensively investigated for various classes of analytic functions and in different functional settings. In particular, a natural question is whether the Bohr inequality involving the term $|f(z)|$ can be refined by replacing it with the stronger quantity $|f(z)|^2$. Addressing this problem, Liu et. al. \cite{LSX2018} established sharp Bohr-type inequalities for analytic functions in the unit disk involving higher-order derivatives. Their results not only provide a refinement of the classical inequality but also determine the corresponding optimal Bohr radii. The following theorems summarize their main findings.
\begin{theoA}\cite[Theorem 2.2]{LSX2018}
Suppose that $N\geq 2$ is an integer, and let $f(z)=\sum_{k=0}^{\infty} a_k z^k$ be analytic in $\mathbb{U}$ with $|f(z)|<1$ in $\mathbb{U}$. Then
\[
|f(z)|+\sum_{k=N}^{\infty}\left|\frac{f^{(k)}(z)}{k!}\right||z|^k\leq 1,
\]
for $|z|=r\leq R_N$, where $R_N$ is the smallest positive root of $\psi_N(r)=(1+r)(1-2r)(1-r)^{N-1}-2r^N=0.$
Moreover, the radius $R_N$ is best possible.
\end{theoA}

\begin{theoB}\cite[Corollary 2.3]{LSX2018}
Suppose that $f(z)=\sum_{k=0}^{\infty} a_k z^k$ is analytic in $\mathbb{U}$ and satisfies $|f(z)|<1$ in $\mathbb{U}$. Then
\[
|f(z)|^2+\sum_{k=N}^{\infty}\left|\frac{f^{(k)}(z)}{k!}\right|
|z|^k\leq 1,
\]
for $|z|=r\leq R_N^{\prime}$, where $R_N^{\prime}$ is the positive root of $(1+r)(1-2r)(1-r)^{N-1}-r^N=0.$ 

Moreover, the radius $R_N^{\prime}$ is best possible.
\end{theoB}
In 2023, Ahamed and Ahammed \cite{AA2023} have obtained the following refined version of Bohr inequality.
\begin{theoC}\cite[Theorem 2.3]{AA2023}
Let $f(z)=\sum_{n=0}^{\infty} a_n z^n $ be an analytic in $\mathbb{U}$ with $|f(z)|<1$ for $z\in \mathbb{U}$.
Then
\begin{equation}\label{O1}
\begin{aligned}
D_f^{*}(z,r)
:=\;& |f(z)|+|f'(z)|\,r+\frac{|f''(z)|}{2}\,r^2+\sum_{n=3}^{\infty}|a_n|r^n  \\
&+\frac{|a_1|^2r^3}{1-r}+\left(\frac{1}{1+|a_0|}+\frac{r}{1-r}\right)\sum_{n=2}^{\infty}|a_n|^2r^{2n}\leq 1,
\end{aligned}
\end{equation}
for $|z|=r\le r_0\approx 0.287459$, where $r_0$ is the unique root of the equation $-1+3r+r^2+r^3+4r^4+2r^5=0$ in $(0,1)$.

Moreover, the constant $r_0\approx 0.287459$ is best possible.
\end{theoC}
Motivated by the above result, it is natural to ask whether inequality \eqref{O1} admits a sharper form when the term $|f(z)|$ is replaced by $|f(z)|^2$.

\medskip

\noindent\textbf{Question.} Can we establish a refinement of inequality \eqref{O1} by replacing $|f(z)|$ with $|f(z)|^2$?

Recently, Ahamed and Roy \cite{AR2025} answered this question affirmatively and obtained the following refined Bohr inequality.
\begin{theoD}\cite[Theorem 2.1.]{AR2025}
Let $f(z)=\sum_{n=0}^{\infty} a_n z^n $ be an analytic in $\mathbb{U}$ with $|f(z)|<1$ for $z\in \mathbb{U}$.
Then
\begin{equation}
\begin{aligned}
D_f(z,r)
:=\;& |f(z)|^2 + |f'(z)|\,r
+\frac{|f''(z)|}{2!}\,r^2
+\sum_{n=3}^{\infty}|a_n|r^n  \\
&+\frac{|a_1|^2 r^3}{1-r}
+\left(\frac{1}{1+|a_0|}+\frac{r}{1-r}\right)\sum_{n=2}^{\infty}|a_n|^2r^{2n}
\leq 1
\end{aligned}
\end{equation}
for $|z|=r\leq r_0\approx 0.393727$, where $r_0$ is the unique root in $(0,1)$ of the equation $-1+2r+r^2+2r^4+r^5=0.$

Moreover, $\lim_{a\to 1^-} D_f(z,r)=1$
for the function $f_a$ defined by
\begin{equation}
f_a(z)
=\frac{a-z}{1-az}
= a-(1-a^2)\sum_{k=1}^{\infty} a^{k-1} z^k,
\qquad z\in\mathbb{D},
\quad a\in(0,1).
\end{equation}
\end{theoD}

It is natural to pose the following problems.

\begin{que1}
Do the refined Bohr inequalities of Theorems C and D admit multidimensional analogues for holomorphic mappings $F:B_X\to \mathbb{U}$
 in complex Banach spaces when combined with one or several Schwarz functions? In particular, can one establish sharp inequalities involving $|F(z)|$ or $|F(z)|^2$, together with suitable derivative and coefficient terms, and determine the corresponding best possible Bohr radii?
\end{que1}
\begin{que2}
Do the derivative Bohr inequalities of Theorems A and B admit analogues for holomorphic mappings $F:B_X\to \mathbb{U}$ between complex Banach spaces in association with Schwarz functions? More precisely, can one establish sharp inequalities involving suitable norms of higher-order Fr\'{e}chet derivatives of $F$, both in the forms corresponding to $|F(z)|$ and its refinement $|F(z)|^2$, and determine the associated optimal Bohr radii?
\end{que2}

The main purpose of this paper is to answer Questions 1.1 and 1.2 in the affirmative by establishing sharp Bohr-type inequalities for holomorphic mappings associated with Schwarz functions in Banach spaces.

\section{{\bf Auxiliary Lemmas.}} 
The following are key lemmas of this paper and will be used to prove the main results.
\begin{lem}\label{L1} \cite{CHPV}
Suppose that $B_{X}$ and $B_{Y}$ are the unit balls of the complex Banach spaces $X$ and $Y$, respectively. Let $f: B_{X} \to B_{Y}$ be a holomorphic mapping. Then
\begin{equation*}
\|f(z)\|_{Y} \le \frac{\|f(0)\|_{Y} + \|z\|_{X}}{1 + \|f(0)\|_{Y}\|z\|_{X}}, \qquad z \in B_{X}.
\end{equation*}
This estimate is sharp with equality possible for each value of $\|f(0)\|_{Y}$ and for each $z \in B_{X}$.
\end{lem}
\begin{lem}\label{L20}\cite{DP2008}, \cite[Theorem 2]{R1985} Suppose that $f$ is an analytic self-maps of the unit disk $\mathbb{U}$.  Then for all $k=1, 2, 3,....$ we have 
\[|f^{(k)}(z)|\leq \frac{k!(1-|f(z)|^2)}{(1-|z|^2)^k}(1+|z|)^{k-1},\;|z|<1.\]
\end{lem}
\begin{lem}\label{L2}\cite{LLP2021}
Suppose that $f$ is an analytic self-maps of the unit disk $\mathbb{U}$. Then for any $N\in \mathbb{N}$, the  following inequality holds
\[\sum_{n=N}^{\infty} |a_n|\, r^{n}+\text{sgn}(t)\sum_{n=1}^t|a_n|^2\frac{r^N}{1-r}+ \left( \frac{1}{1+|a_0|}+ \frac{r}{1-r} \right)\sum_{n=t+1}^{\infty} |a_n|^{2} r^{2n}\le (1-|a_0|^{2})\,\frac{r^N}{1-r},\]
for $r\in [0,1)$, where $t=\lfloor \frac{N-1}{2}\rfloor$.
\end{lem}

\begin{lem}\label{L3}
Let $B_{X}$ be the unit ball of a complex Banach space $X$ also let $F$ be a holomorphic mapping from $B_{X}$ to $\ol{\mathbb{U}}$ with
\[
F(z)=a+\sum_{s=1}^{\infty}P_{s}(z), \qquad z\in B_{X},
\]
where $P_{s}(z)=\frac{1}{s!}D^{s}F(0)(z^{s})$ and   $\mu_k,\mu_m:B_{X}\to B_{X}$ are  Schwarz mappings having $z=0$ as a zero of order $k,m$ respectively. Then, for each $k\in\mathbb{N}$, the inequality
\begin{equation}
\sum_{j=3}^{\infty} |P_j(\mu_k(z))| +\frac{|P_1(\mu_k(z))|^2r^{k}}{1-r^k}
+\left(\frac{1}{1+|a|}+\frac{r^k}{1-r^k}\right)
\sum_{j=2}^{\infty} |P_j(\mu_k(z))|^2
\le (1-|a|^2)\frac{r^{3k}}{1-r^k}
\label{eq:lemma3}
\end{equation}
holds for $\|z\|=r\in[0,1)$.
\end{lem}

\begin{proof}
Let $z\in B_X\setminus\{0\}$ be fixed and define $z_0=\frac{z}{\|z\|_X}\in \partial B_X.$ Also let $f(\zeta)=F(\zeta z_{0}), \zeta\in\mathbb{\ol U}.$
 Then $f$ are the holomorphic functions in $\mathbb{U}$ and
\[
f(\zeta)=a+\sum_{s=1}^{\infty}P_{s}(z_{0})\zeta^{s},\qquad \zeta\in\mathbb{U}.
\]
 We have (see e.g. \cite[p.\,35]{GK2003})
\bea\label{l0}
|P_{s}(z_{0})| \le 1-|a|^{2}
\eea
and  applying Lemma \ref{L2}, we get
\bea\label{l1} &&\sum_{s=3}^{\infty}|(P_{s})(z_{0})|\,|\zeta|^s+|P_1(z_0)|^2\frac{|\zeta|^3}{1-|\zeta|}
+\left(\frac{1}{1+|a|}+\frac{|\zeta|}{1-|\zeta|}\right)
\sum_{s=2}^{\infty} |P_{s}(z_{0})|^2 |\zeta|^{2s}\nonumber\\
& \le & (1-|a|^2)\frac{|\zeta|^3}{1-|\zeta|}.
\eea
Let $\zeta=\|z\|_{X}=r<1$. Then, by the above equation and $(\ref{l1}),$ it follows that
\bea\label{hhh1} &&\sum_{s=3}^{\infty}|P_s(z)| +|P_1(z)|^2\frac{\|z\|_X}{1-\|z\|_X}+\left(\frac{1}{1+|a|}+\frac{\|z\|_X}{1-\|z\|_X}\right)
\sum_{s=2}^{\infty}  |P_s(z)|^{2}\nonumber\\
&\le& (1-|a|^2)\frac{\|z\|_X^3}{1-\|z\|_X}.\eea
Let $\phi_1(x)=\frac{x}{1-x}$. Then $\phi_1'(x)=\frac{1}{(1-x)^2}>0$ and so $\phi_1(x)$ is increasing on $(0,1)$. Since $r^k \le r$ for $k \ge 1$, it follows that
\[
\phi(r^k) \le \phi(r).
\]
Hence,
\[
\frac{r^k}{1-r^k} \le \frac{r}{1-r} = \frac{\|z\|_X}{1-\|z\|_X}.
\]
Therefore 
\bea\label{aq1}&&\sum_{s=3}^{\infty}|P_s(z)| +|P_1(z)|^2\frac{r^{k}}{1-r^k}+\left(\frac{1}{1+|a|}+\frac{r^k}{1-r^k}\right)
\sum_{s=2}^{\infty}  |P_s(z)|\nonumber\\
&\leq& \sum_{s=3}^{\infty}|P_s(z)|+|P_1(z)|^2\frac{\|z\|_X}{1-\|z\|_X} +\left(\frac{1}{1+|a|}+\frac{\|z\|_X}{1-\|z\|_X}\right)
\sum_{s=2}^{\infty}  |P_s(z)|^{2}.\eea
In view of  (\ref{eq:Schwarz}), (\ref{hhh1}) and (\ref{aq1}), we obtain  
\[\sum_{s=3}^{\infty}|P_s(\mu_k(z))|+|P_1(\mu_k(z))|^2\frac{r^{k}}{1-r^k} +\left(\frac{1}{1+|a|}+\frac{r^{k}}{1-r^k}\right)
\sum_{s=2}^{\infty}  |P_s(\mu_k(z))|^{2} \leq (1-|a|^2)\frac{r^{3k}}{1-r^k}.\]
\end{proof}

\section{{\bf Main results and their proofs.}}
In the setting of holomorphic mappings in Banach spaces, the Bohr phenomenon exhibits several interesting features when combined with Schwarz functions. The following result provides a sharp refined version of the Bohr inequality involving several Schwarz functions and establishes the corresponding best possible Bohr radius.
\begin{theo}\label{T1}
Let $B_{X}$ be the unit ball of a complex Banach space $X$ also let $F$ be a holomorphic mapping from $B_{X}$ to $\ol{\mathbb{U}}$ with
\[
F(z)=a+\sum_{s=1}^{\infty}P_{s}(z), \qquad z\in B_{X},
\]
where $P_{s}(z)=\frac{1}{s!}D^{s}F(0)(z^{s})$ and   $\mu_k,\mu_m:B_{X}\to B_{X}$ are  Schwarz mappings having $z=0$ as a zero of order $k,m$ respectively. Then, for each $k,m\in\mathbb{N}$, the inequality
\bea\label{T1.1}&&|F(\mu_m(z))|+|D F(\mu_m(z))\mu_m(z)|+\frac{1}{2!}|D^2 F(\mu_m(z)) ((\mu_m(z))^2)|\nonumber\\
&&+\sum_{j=3}^{\infty} |P_j(\mu_k(z))| +\frac{|P_1(\mu_k(z))|^2r^{k}}{1-r^k}+\left(\frac{1}{1+|a|}+\frac{r^k}{1-r^k}\right)
\sum_{j=2}^{\infty} |P_j(\mu_k(z))|^2\leq 1,
\eea
for $\|z\|=r\leq R_{1,m,k}$, where $R_{1,m}\in(0,1)$ is the unique root of the equation $r^{3m}-r^{2m}-3r^m+1=0$
and $R_{1,m,k}\in(0,1)$ is the smallest positive root of the equation
\[
(1-r^k)\bigl(-1+3r^m+r^{2m}-r^{3m}\bigr)+2r^{3k}(1-r^m)(1+r^m)^2=0.
\]
The number $R_{1,m,k}$ cannot be improved.
\end{theo}

\begin{proof}[{\bf Proof of Theorem \ref{T1.1}}]
Let $z\in B_X\setminus\{0\}$ be fixed and define $z_0=\frac{z}{\|z\|_X}\in \partial B_X.$
We consider the holomorphic function $f: \mathbb{U} \to \overline{\mathbb{U}}$ defined by
\[
f(\zeta) = F(\zeta z_0), \qquad \zeta \in \mathbb{U}.
\]
 Then 
\begin{equation}\label{eq:fexp}
f(\zeta)=a_j+\sum_{s=1}^{\infty}(P_s)(z_0)\zeta^s,
\qquad \zeta\in\mathbb{U}.
\end{equation}
Hence, we get (see, e.g., \cite[p.~35]{GK2003}),
$|(P_s)(z_0)|\le 1-|a|^2$.

Differentiating $f(\zeta)= F(\zeta z_0)$, we have 
\[f'(\zeta)=D F(\zeta z_0) z_0\;\text{and}\; f''(\zeta)=D^2 F(\zeta z_0) (z_0^2).\]
Then by Schwarz-Pick Lemma, we can conclude that 
\bea\label{hkl1}|DF(\zeta z_0)z_0|=|f'(\zeta)|\leq \frac{1-|f(\zeta)|^2}{1-|\zeta|^2}\;\text{and}\; |D^2F(\zeta z_0)(z_0)^2|\leq \frac{2(1-|f(\zeta)|^2)}{(1-|\zeta|^2)^2}(1+|\zeta|).\eea

Setting $\zeta=\|z\|_X=r<1$ then above equation \ref{hkl1}. become 
\bea\label{dd1} |DF(z) z|\leq \frac{1-|F(z)|^2}{1-\|z\|_X^2}\|z\|_X \;\text{and}\;|D^2F(z) (z^2)|\leq \frac{2(1-|F(z)|^2)}{(1-\|z\|_X^2)^2}(1+\|z\|_X)\|z\|_X^2.\eea

Let
\[
t_1(x)=\frac{a+x}{1+ax}, \qquad
t_2(x)=\frac{x}{1-x^2}, \qquad
t_3(x)=\frac{2(1+x)x^2}{(1-x^2)^2}.
\]
A straightforward computation yields
\[
t_1'(x)=\frac{1-a^2}{(1+ax)^2}, \qquad
t_2'(x)=\frac{1+x^2}{(1-x^2)^2},
\]
and
\[
t_3'(x)=\frac{2x\left(2+3x+4x^2+x^3\right)}{(1-x^2)^3}.
\]
Since \(0\le a<1\) and \(0\le x<1\), we have $t_i'(x)\ge 0, i=1,2,3.$

Therefore, each \(t_i\) is increasing on \([0,1)\).

Since \(\|\mu_m(z)\|_X\leq \|z\|_X^m<1\), it follows from \eqref{dd1} and the monotonicity of \(t_1\), \(t_2\), and \(t_3\) that
\begin{align}
|F(\mu_m(z))|
&\leq
\frac{a+\|\mu_m(z)\|_X}{1+a\|\mu_m(z)\|_X}
\leq
\frac{a+\|z\|_X^m}{1+a\|z\|_X^m},
\label{eq:est1}
\\[2mm]
|DF(\mu_m(z))(\mu_m(z))|
&\leq
\frac{1-|F(\mu_m(z))|^2}
     {1-\|\mu_m(z)\|_X^2}
\,\|\mu_m(z)\|_X
\leq
\frac{1-|F(\mu_m(z))|^2}
     {1-\|z\|_X^{2m}}
\,\|z\|_X^m,
\label{eq:est2}
\\[2mm]
|D^2F(\mu_m(z))(\mu_m(z)^2)|
&\leq
\frac{2(1-|F(\mu_m(z))|^2)}
     {(1-\|\mu_m(z)\|_X^2)^2}
(1+\|\mu_m(z)\|_X)\,
\|\mu_m(z)\|_X^2
\nonumber\\
&\leq
\frac{2(1-|F(\mu_m(z))|^2)}
     {(1-\|z\|_X^{2m})^2}
(1+\|z\|_X^m)\,
\|z\|_X^{2m}.
\label{eq:est3}
\end{align}
Let $x=|a|\in[0,1]$. If $x=1$, then $P_s(z_0)=0$ for all $s\ge1$, and
\eqref{T1.1} holds trivially. 

Hence, we assume $x\in[0,1)$. Using  \eqref{eq:est1}- \eqref{eq:est3}, Lemma \ref{L1} and \ref{L3}, we obtain
\begin{align}
\mathcal{A}_F(z):=&|F(\mu_m(z))|+|D F(\mu_m(z))\mu_m(z)|+\frac{1}{2!}|D^2 F(\mu_m(z)) ((\mu_m(z))^2)|\nonumber\\
&+\sum_{j=3}^{\infty} |P_j(\mu_k(z))| +\frac{|P_1(\mu_k(z))|^2r^{k}}{1-r^k}
+\left(\frac{1}{1+|a|}+\frac{r^k}{1-r^k}\right)
\sum_{j=2}^{\infty} |P_j(\mu_k(z))|^2\nonumber\\ 
\le & |F(\mu_m(z))|+\frac{r^m}{1-r^{2m}}(1-|F(\mu_m(z))|^2)+\frac{r^{2m}(1+r^m)}{(1-r^{2m})^2}(1-|F(\mu_m(z))|^2)\nonumber\\
&+(1-x^2)\frac{r^{3k}}{1-r^k} \nonumber\\
=& \frac{x+r^m}{1+xr^m}+\left(\frac{r^m}{1-r^{2m}}+\frac{r^{2m}(1+r^m)}{(1-r^{2m})^2}\right)\bigg(1-\left(\frac{x+r^m}{1+xr^m}\right)^2\bigg)+(1-x^2)\frac{r^{3k}}{1-r^k}\nonumber\\
=& \frac{x+r^m}{1+xr^m}+\frac{r^m(1-x^2)}{(1-r^m)(1+xr^m)^2}+(1-x^2)\frac{r^{3k}}{1-r^k}\nonumber\\
=&\frac{\begin{aligned}& (x+r^m)(1-r^k)(1-r^m)(1+xr^m)+r^m(1-x^2)(1-r^k)\\
&\quad+(1-x^2)r^{3k}(1-r^m)(1+xr^m)^2\end{aligned}}{(1-r^k)(1-r^m)(1+xr^m)^2}:=G_1(x,r).\label{eq:G1}
\end{align}
Now $G_1(x,r)\leq 1$ if $\Phi_1(x,r)\leq 0$, where $\Phi_1(x,r)$ is defined by
\[
\Phi_1(x,r)=-(1-r^k)(1-r^m)^2(1+xr^m) +(1+x)r^m(1-r^k)+(1+x)r^{3k}(1-r^m)(1+xr^m)^2.
\]
It is worth pointing out that, in the third inequality of Eq. (\ref{eq:G1}), we have used the fact that 
\bea\label{kkkk1} \theta (t)=t+\lambda (1-t^2)\leq \theta(t_0),\eea
whenever \[t=F(\mu_m(z))\leq t_0=\frac{x+r^m}{1+xr^m} \;\text{and}\; \lambda=\frac{r^m}{(1-r^{m})^2(1+r^m)}.\]
Inequality (\ref{kkkk1}) holds provided that $\frac{r^m}{(1-r^{m})^2(1+r^m)}\leq \frac{1}{2}$, which is satisfied for $0\leq r\leq R_{1,m}$, where $R_{1,m}$ is the unique positive root of the equation 
\bea\label{R1m}r^{3m}-r^{2m}-3r^m+1=0.\eea

Now differentiating partially with respect to $x$, we get
\[
\frac{\partial \Phi_1(x,r)}{\partial x}=r^{2m}(1-r^k)(2-r^m)+r^{3k}(1-r^m)(1+xr^m)((1+3xr^m+2r^m)\geq 0.
\]
Therefore $\Phi_1(x,r)$ is a monotonically increasing function of $x\in [0,1)$ and hence, we gave 
\[\Phi_1(x,r)\leq \Phi_1(1,r)=(1-r^k)\bigl(-1+3r^m+r^{2m}-r^{3m}\bigr)+2r^{3k}(1-r^m)(1+r^m)^2.\]
Thus $\mathcal{A}_F(z)\le1$ whenever $G_1(1,r)\le0$, which holds for
$r\le R_{1,m,k}$, where $R_{1,m,k}$ is minimal positive root in $(0,1)$ of
\begin{equation}\label{eq:R1}
(1-r^k)\bigl(-1+3r^m+r^{2m}-r^{3m}\bigr)+2r^{3k}(1-r^m)(1+r^m)^2=0.
\end{equation}

We now claim that $R_{1,m,k}<R_{1,m}$. Suppose, on the contrary that $R_{1,m,k}\geq R_{1,m}$. Then for any $r\geq R_{1,m}$, we have $r^{3m}-r^{2m}-3r^m+1\leq  0$ and hence 
\beas (1-r^k)\bigl(-1+3r^m+r^{2m}-r^{3m}\bigr)+2r^{3k}(1-r^m)(1+r^m)^2>0.\eeas
This contradicts $R_{1,m,k}$ roots of (\ref{eq:R1}). Therefore $R_{1,m,k}<R_{1,m}$.

\medskip

{\bf \underline{Sharpness of the Radius $R_{1,m,k}$}: }

We now show that the radius $R_{1,m,k}$ is sharp.
Fix $z_0\in\partial B_X$.
For $b\in(0,1)$, define
\begin{equation}\label{F1}
F_1(z)=f(l_{z_0}(z)),\qquad z\in B_X,
\end{equation}
where
\[
f(\zeta)=\frac{b+\zeta}{1+b\zeta},\qquad \zeta\in\mathbb{U},
\]
and $l_{z_0}\in T(z_0)$.
Let $\mu_m(z)=l_{z_0}(z)^{\,m-1}z$ for $m\geq 1$.
Then, for $r\in(0,1)$,
\[
F_1(rz_0)=\frac{b+r}{1+br}
= a+\sum_{s=1}^{\infty}P_s(rz_0),
\]
where  $P_s(z_0)=(-1)^sb^{\,s-1}(1-b^2)r^s, s\ge 1.$
Consequently,
\[
P_s(\mu_m(rz_0))
=-(-1)^sr^{ms}b^{\,s-1}(1-b^2).
\]
Moreover 
\[DF_1(z)(z)=\frac{(1-b^2)\,l_{z_0}(z)}{\bigl(1+b\,l_{z_0}(z)\bigr)^2}\;\;\text{and}\;\;D^2F_1(z)(z^2)=\frac{2b(1-b^2)\,[l_{z_0}(z)]^2}{\bigl(1+b\,l_{z_0}(z)\bigr)^3}.\]
Thus for $z=rz_0$, we have 
\begin{align*}
\mathcal{A}_{F_1}(z):=&|F_1(\mu_m(rz_0))|+|D F_1(\mu_m(rz_0))\mu_m(rz_0)|+\frac{1}{2!}|D^2 F_1(\mu_m(rz_0)) ((\mu_m(rz_0))^2)|\\
&+\sum_{j=3}^{\infty}|P_j(\mu_k(rz_0))|+\frac{|P_1(\mu_k(rz_0))|^2r^{k}}{1-r^k}+\left(\frac{1}{1+|b|}+\frac{r^k}{1-r^k}\right)\sum_{j=2}^{\infty}|P_j(\mu_k(rz_0))|^2 \\
=&\frac{b+r^m}{1+br^m}+\frac{(1-b^2)r^m}{(1+br^m)^2}+\frac{b(1-b^2)r^{2m}}{(1+br^m)^3}+\frac{(1-b^2)b^2r^{3k}}{1-br^k}+\frac{(1-b^2)^2r^{3k}}{1-r^k}\\
&+\left(\frac{1}{1+|b|}+\frac{r^k}{1-r^k}\right)\frac{(1-b^2)^2\,r^{2k}}{1-b^2r^{2k}}\\
&=1+\frac{(1-b)D_1(b,r)}{1+br^m}
\end{align*}
where
\beas
D_1(b,r)
&=&-(1-r^m)+\frac{(1+b)r^m}{1+br^m}+\frac{b(1+b)r^{2m}}{(1+br^m)^2}+\frac{(1+b)b^2(1+br^m)r^{3k}}{1-br^k}\\
&&+\frac{(1-b^2)(1+b)(1+br^m)r^{3k}}{1-r^k}+\left(\frac{1}{1+|b|}+\frac{r^k}{1-r^k}\right)\frac{(1-b^2)(1+b)(1+br^m)\,r^{2k}}{1-b^2r^{2k}}.
\eeas
Letting $b\to1^-$, we obtain
\beas
&&\lim_{b\to1^-}D_1(b,r)\\
&=&-(1-r^m)+\frac{2r^m}{1+r^m}+\frac{2r^{2m}}{(1+r^m)^2}+\frac{2(1+r^m)r^{3k}}{1-r^k}\\
&=&\frac{-(1-r^m)(1+r^m)^2(1-r^k)+2r^m(1+r^m)(1-r^k)+2r^{2m}(1-r^k)+2(1+r^m)^3r^{3k}}{(1+r^m)^2(1-r^k)}\\
&=&\frac{(1-r^k)\bigl(-1+3r^m+r^{2m}-r^{3m}\bigr)+2r^{3k}(1-r^m)(1+r^m)^2+4r^{m+3k}(1+r^m)^2}{(1+r^m)^2(1-r^k)}>0.\eeas
Hence, for $r>R_{1,m,k}$ and $b$ sufficiently close to $1$, we have
$\mathcal{A}_{F_1}(z)>1$.
This proves that the radius $R_{1,m,k}$ is best possible.
\end{proof}
\begin{rem}
Theorem \ref{T1} reduces to Theorem C in \cite{AA2023} when $B_X=\mathbb{U}$, $m=k=1$ and $\mu_1(z)=z$.
\end{rem}
In table \ref{tab3.1}, Figure \ref{f31} and \ref{f32}, we obtain the values of $R_{1,m,k}$ and $R_{1,m}$ for certian values of $m,k\in\mathbb{N}$.

\begin{table}[H]
\centering
\begin{minipage}[b]{0.45\textwidth}
\centering
\begin{tabular}{|c|c|c|c|}
\hline
\textbf{k} & \textbf{m} &\textbf{$R_{1,m,k}$} & \textbf{$R_{1,m}$} \\ 
\hline
1 & 1 & 0.287460& 0.311110 \\ 
\hline
2 & 2 & 0.536150& 0.557770 \\ 
\hline
2 & 3 & 0.624040 & 0.677600 \\ 
\hline
3 & 4 & 0.712960& 0.746840 \\ 
\hline
\end{tabular}
\caption{$R_{1,m}$ and $R_{1,m,k}$ are the smallest root of equations (\ref{R1m}) and(\ref{eq:R1}) respectively in $(0,1)$.}
\label{tab3.1}
\end{minipage}%
\hfill
\begin{minipage}[b]{0.55\textwidth}
\centering
\begin{figure}[H]
\centering
\includegraphics[width=\textwidth]{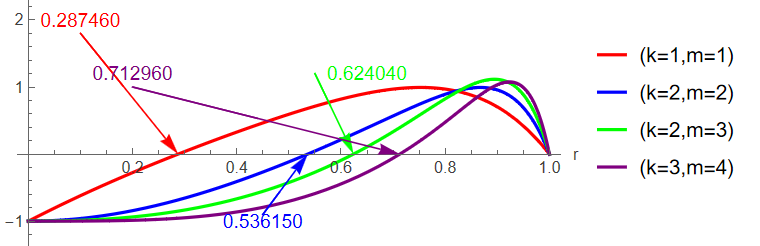}
\caption{The graphs exhibit the locations of the roots $R_{1,m,k}$ in $(0,1)$ for different values of $k,m$.}
\label{f31}
\end{figure}
\end{minipage}
\begin{minipage}[b]{0.55\textwidth}
\centering
\begin{figure}[H]
\centering
\includegraphics[width=\textwidth]{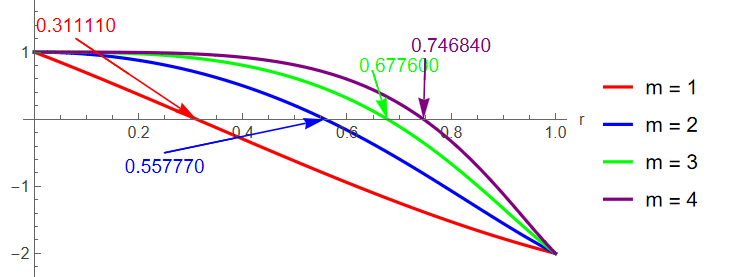}
\caption{The graphs exhibit the locations of the roots $R_{1,m}$ in $(0,1)$ for different value of $m$.}
\label{f32}
\end{figure}
\end{minipage}

\end{table}

\begin{theo}\label{T2}
Let $B_{X}$ be the unit ball of a complex Banach space $X$ also let $F$ be a holomorphic mapping from $B_{X}$ to $\ol{\mathbb{U}}$ with
\[
F(z)=a+\sum_{s=1}^{\infty}P_{s}(z), \qquad z\in B_{X},
\]
where $P_{s}(z)=\frac{1}{s!}D^{s}F(0)(z^{s})$ and   $\mu_k,\mu_m:B_{X}\to B_{X}$ are  Schwarz mappings having $z=0$ as a zero of order $k,m$ respectively. Then, for each $k,m\in\mathbb{N}$, the inequality
\bea\label{T2.2}&&\mathcal{A}_F(z,r):=\nonumber\\
&&|F(\mu_m(z))|^2+|D F(\mu_m(z))\mu_m(z)|+\frac{1}{2!}|D^2 F(\mu_m(z)) ((\mu_m(z))^2)|\nonumber\\
&&+\sum_{j=3}^{\infty} |P_j(\mu_k(z))| +\frac{|P_1(\mu_k(z))|^2r^{k}}{1-r^k}+\left(\frac{1}{1+|a|}+\frac{r^k}{1-r^k}\right)
\sum_{j=2}^{\infty} |P_j(\mu_k(z))|^2\leq 1,
\eea
for $\|z\|=r\leq R_{2,m,k}$, where $R_{2,m}\in(0,1)$ is the unique root of the equation $r^{3m}-r^{2m}-2r^m+1=0$
and $R_{2,m,k}\in(0,1)$ is the smallest positive root of the equation
\[
(1-r^k)\bigl(-1+3r^m+r^{2m}-r^{3m}\bigr)+2r^{3k}(1-r^m)(1+r^m)^2=0.
\]
Moreover, $\lim_{a\to 1^-} \mathcal{A}_F(z,r)=1$
for the function $F_2$ defined in (\ref{F2}).

\end{theo}

\begin{proof}[{\bf Proof of Theorem \ref{T2}}] Assume that $|a|=x\in[0,1]$. If $x=1$, then $P_s(z_0)=0$ for all $s\ge1$, and
\eqref{T2.2} holds trivially. 

Hence, we assume $x\in[0,1)$. Using \eqref{eq:est1}- \eqref{eq:est3}, Lemma \ref{L1} and \ref{L3}, we obtain
\begin{align}
\mathcal{A}_F(z):=&|F(\mu_m(z))|^2+|D F(\mu_m(z))\mu_m(z)|+\frac{1}{2!}|D^2 F(\mu_m(z)) ((\mu_m(z))^2)|\nonumber\\
&+\sum_{j=3}^{\infty} |P_j(\mu_k(z))| +\frac{|P_1(\mu_k(z))|^2r^{k}}{1-r^k}
+\left(\frac{1}{1+|a|}+\frac{r^k}{1-r^k}\right)
\sum_{j=2}^{\infty} |P_j(\mu_k(z))|^2\nonumber\\ 
\le & |F(\mu_m(z))|^2+\frac{r^m}{1-r^{2m}}(1-|F(\mu_m(z))|^2)+\frac{r^{2m}(1+r^m)}{(1-r^{2m})^2}(1-|F(\mu_m(z))|^2)\nonumber\\
&+(1-x^2)\frac{r^{3k}}{1-r^k} \nonumber\\
=& \left(\frac{x+r^m}{1+xr^m}\right)^2+\left(\frac{r^m}{1-r^{2m}}+\frac{r^{2m}(1+r^m)}{(1-r^{2m})^2}\right)\bigg(1-\left(\frac{x+r^m}{1+xr^m}\right)^2\bigg)+(1-x^2)\frac{r^{3k}}{1-r^k}\nonumber\\
=& \left(\frac{x+r^m}{1+xr^m}\right)^2+\frac{r^m(1-x^2)}{(1-r^m)(1+xr^m)^2}+(1-x^2)\frac{r^{3k}}{1-r^k}\nonumber\\
=& 1+\frac{(1-x^2)G_2(x,r)}{(1-r^k)(1-r^m)(1+xr^m)^2},\label{eq:G2}
\end{align}
where $G_2(x,r)=-(1-r^{2m})(1-r^m)(1-r^k)+r^m(1-r^k)+r^{3k}(1+xr^m)^2(1-r^m)$.\\

The third inequality of Eq. (\ref{eq:G2}) holds for $r\in [0,1]$ satisfying $\frac{r^m}{(1-r^{m})^2(1+r^m)}\leq 1$, which is satisfied for $0\leq r\leq R_{2,m}$, where $R_{2,m}$ is the unique positive root of the equation $r^{3m}-r^{2m}-2r^m+1=0$.

Differentiating $G_2(x,r)$ with respect to $x$, we obtain
\[
\frac{\partial G_2(x,r)}{\partial x}
=2r^{3k+m}(1+xr^m)(1-r^m)\ge 0.
\]
Consequently, $G_2(x,r)$ is increasing in $x\in[0,1)$. Therefore,
\[
G_2(x,r)\le G_2(1,r)
=-(1-r^{2m})(1-r^m)(1-r^k)+r^m(1-r^k)+r^{3k}(1+r^m)^2(1-r^m).
\]
It follows that $\mathcal{A}_F(z)\le 1$ whenever $G_2(1,r)\le 0$. This condition is satisfied for  $r\le R_{2,m,k},$
where $R_{2,m,k}$ denotes the smallest positive root in $(0,1)$ of
\begin{equation}\label{eq:R2}
-(1-r^{2m})(1-r^m)(1-r^k)+r^m(1-r^k)+r^{3k}(1+r^m)^2(1-r^m)=0.
\end{equation}

Next, we show that $R_{2,m,k}<R_{2,m}$. Assume, to the contrary, that
$R_{2,m,k}\ge R_{2,m}$. Since $R_{2,m}$ is the unique positive root of
\bea\label{R2m}
r^{3m}-r^{2m}-2r^m+1=0,
\eea
we have
\[
-1+2r^m+r^{2m}-r^{3m}\ge 0
\qquad \text{for } r\ge R_{2,m}.
\]
Hence, for every $r\ge R_{2,m}$,
\beas
&&-(1-r^{2m})(1-r^m)(1-r^k)+r^m(1-r^k)+r^{3k}(1+r^m)^2(1-r^m)\\
&=&(1-r^k)\bigl(-1+2r^m+r^{2m}-r^{3m}\bigr)+r^{3k}(1-r^m)(1+r^m)^2> 0.
\eeas
This contradicts the fact that $R_{2,m,k}$ is a root of \eqref{eq:R2}. Therefore,
\[
R_{2,m,k}<R_{2,m}.
\]

{\bf \underline{Sharpness of the Radius $R_{2,m,k}$}: }

We now show that the radius $R_{2,m,k}$ is sharp. Fix $z_0\in\partial B_X$. For $b\in(0,1)$, define
\begin{equation}\label{F2}
F_2(z)=f(l_{z_0}(z)),\qquad z\in B_X,
\end{equation}
where
\[
f(\zeta)=\frac{b-\zeta}{1-b\zeta},\qquad \zeta\in\mathbb{U},
\]
and $l_{z_0}\in T(z_0)$. Let $\mu_m(z)=l_{z_0}(z)^{m-1}z, m\geq 1.$ Then, for $r\in(0,1)$,
\[
F_2(rz_0)=\frac{b-r}{1-br}
=b+\sum_{s=1}^{\infty}P_s(rz_0),
\]
where $P_s(z_0)=-(1-b^2)b^{s-1}, s\ge1.$ Consequently, $P_s(\mu_m(rz_0))=-(1-b^2)b^{s-1}r^{ms}.$

Moreover,
\[
DF_2(z)(z)
=\frac{(1-b^2)l_{z_0}(z)}
{\bigl(1-bl_{z_0}(z)\bigr)^2}, \;\text{and}\; D^2F_2(z)(z^2)
=\frac{2b(1-b^2)[l_{z_0}(z)]^2}
{\bigl(1-bl_{z_0}(z)\bigr)^3}.
\]

Thus for $z=rz_0$, we have 
\begin{align*}
\mathcal{A}_{F_2}(z,b):=&|F_2(\mu_m(rz_0))|^2+|D F_2(\mu_m(rz_0))\mu_m(rz_0)|+\frac{1}{2!}|D^2 F_2(\mu_m(rz_0)) ((\mu_m(rz_0))^2)|\\
&+\sum_{j=3}^{\infty}|P_j(\mu_k(rz_0))|+\frac{|P_1(\mu_k(rz_0))|^2r^{k}}{1-r^k}+\left(\frac{1}{1+|a|}+\frac{r^k}{1-r^k}\right)\sum_{j=2}^{\infty}|P_j(\mu_k(rz_0))|^2 \\
=&\left(\frac{b-r^m}{1-br^m}\right)^2+\frac{(1-b^2)r^m}{(1-br^m)^2}+\frac{b(1-b^2)r^{2m}}{(1-br^m)^3}+\frac{(1-b^2)b^2r^{3k}}{1-br^k}+\frac{(1-b^2)^2r^{3k}}{1-r^k}\\
&+\left(\frac{1}{1+|b|}+\frac{r^k}{1-r^k}\right)\frac{(1-b^2)^2\,r^{2k}}{1-b^2r^{2k}}\\
&=1+\frac{(1-b^2)D_2(b,r)}{(1-br^m)^2}
\end{align*}
where
\beas
D_2(b,r)
&=&-(1-r^{2m})+r^m+\frac{br^{2m}}{1-br^m}+\frac{b^2(1-br^m)^2r^{3k}}{1-br^k}\\
&&+\frac{(1-b^2)(1-br^m)^2r^{3k}}{1-r^k}+\left(\frac{1}{1+|b|}+\frac{r^k}{1-r^k}\right)\frac{(1-b^2)(1-br^m)^2\,r^{2k}}{1-b^2r^{2k}}
\eeas
Letting $b\to1^-$, we obtain
\beas
&&\lim_{b\to1^-}\mathcal{A}_{F_2}(z,r)=1.\eeas

\end{proof}
\begin{rem}
Theorem \ref{T2} reduces to Theorem D in \cite{AR2025} when $B_X=\mathbb{U}$, $m=k=1$ and $\mu_1(z)=z$.
\end{rem}
In table \ref{tab3.2}, Figure \ref{f33} and \ref{f34}, we obtain the values of $R_{2,m,k}$ and $R_{2,m}$ for certian values of $m,k\in\mathbb{N}$.

\begin{table}[H]
\centering
\begin{minipage}[b]{0.45\textwidth}
\centering
\begin{tabular}{|c|c|c|c|}
\hline
\textbf{k} & \textbf{m} &\textbf{$R_{2,m,k}$} & \textbf{$R_{2,m}$} \\ 
\hline
1 & 1 & 0.393730& 0.445040 \\ 
\hline
2 & 2 & 0.627480& 0.667110 \\ 
\hline
2 & 3 & 0.694260 & 0.763480 \\ 
\hline
3 & 5 & 0.794350& 0.850510 \\ 
\hline
\end{tabular}
\caption{$R_{2,m}$ and $R_{2,m,k}$ are the smallest root of equations (\ref{R2m}) and(\ref{eq:R2}) respectively in $(0,1)$.}
\label{tab3.2}
\end{minipage}%
\hfill
\begin{minipage}[b]{0.55\textwidth}
\centering
\begin{figure}[H]
\centering
\includegraphics[width=\textwidth]{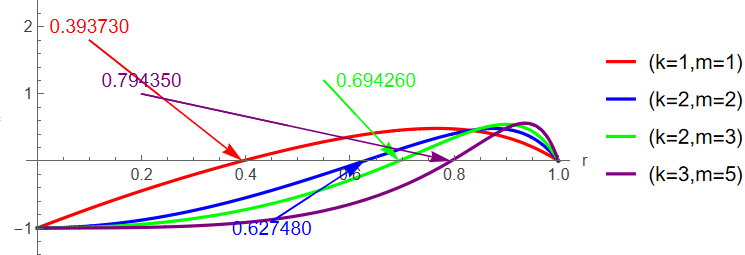}
\caption{The graphs exhibit the locations of the roots $R_{2,m,k}$ in $(0,1)$ for different values of $k,m$.}
\label{f33}
\end{figure}
\end{minipage}
\begin{minipage}[b]{0.55\textwidth}
\centering
\begin{figure}[H]
\centering
\includegraphics[width=\textwidth]{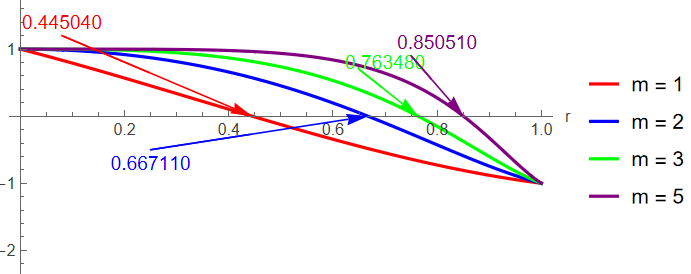}
\caption{The graphs exhibit the locations of the roots $R_{2,m}$ in $(0,1)$ for different value of $m$.}
\label{f34}
\end{figure}
\end{minipage}

\end{table}

\begin{theo}\label{T3}
Let $B_{X}$ be the unit ball of a complex Banach space $X$ also let $F$ be a holomorphic mapping from $B_{X}$ to $\ol{\mathbb{U}}$ with
\[
F(z)=a+\sum_{s=1}^{\infty}P_{s}(z), \qquad z\in B_{X},
\]
where $P_{s}(z)=\frac{1}{s!}D^{s}F(0)(z^{s})$ and   $\mu_m:B_{X}\to B_{X}$ is a Schwarz mappings having $z=0$ as a zero of order $m$. Then, for each $N,m\in\mathbb{N}$, the inequality
\bea\label{T2.3}&&|F(\mu_m(z))|+\sum_{k=N}^{\infty}\bigg|\frac{D^kF(\mu_m(z)) ((\mu_m(z))^k)}{k!}\bigg|\leq 1\
\eea
for $\|z\|=r\leq R_{3,m,N}$, where  $R_{3,m,N}\in(0,1)$ is the smallest positive root of the equation
\[
2r^{mN}-(1-r^{m})^{N-1}(1+r^m)(1-2r^m)=0.
\]
The number $R_{3,m,N}$ cannot be improved.
\end{theo}
\begin{proof}[{\bf Proof of the Theorem \ref{T3} }]
Let $z\in B_X\setminus\{0\}$ be fixed and define $z_0=\frac{z}{\|z\|_X}\in \partial B_X.$
We consider the holomorphic function $f: \mathbb{U} \to \overline{\mathbb{U}}$ defined by
\[
f(\zeta) = F(\zeta z_0), \qquad \zeta \in \mathbb{U}
\]
 Then 
\begin{equation}\label{te31}
f(\zeta)=a_j+\sum_{s=1}^{\infty}(P_s)(z_0)\zeta^s,
\qquad \zeta\in\mathbb{U}.
\end{equation}
Now differentiating $k-$ times Eq. (\ref{te31}), we get 
\[f^{(k)}(\zeta)=D^kF(\zeta z_0)(z_0^k).\]
Then by  Lemma\ref{L20} (Schwarz-Pick Lemma), we can deduce that
\[ |D^kF(\zeta z_0)(z_0^k)|=|f^{(k)}(\zeta)|\leq \frac{k!(1-|f(\zeta)|^2)}{(1-|\zeta|^2)^k}(1+|\zeta|)^{k-1}.\]
Setting $\zeta=\|z\|_X=r<1$ then above equation Eqs. become 
\bea\label{te32} |D^kF(z) (z^k)|\leq \frac{k!(1-|F(z)|^2)}{(1-\|z\|^2)^k}(1+\|z\|)^{k-1}\|z\|^k.\eea
Let $t_3(x)=\frac{(1+x)^{k-1}}{(1-x^2)^k}, 0\le x\le x_0<1,$ where $k\ge 2$. A straightforward computation yields $t_3'(x)=\frac{x^{k-1}\bigl(k+(k-1)x+x^2\bigr)}{(1-x)^{k+1}(1+x)^2}.$
Since $n\ge 2$ and $0\le x<1$, it follows that $t_3'(x)\ge 0.$
Hence $t_3(x)$ is an increasing function on $[0,1]$. Since $\|\mu_m(z)\|_X\le \|z\|_X^{\,m}\leq 1$, then we get 
\[t_3(\|\mu_m(z)\|_X)\leq t_3(\|z\|_X^{\,m})\]
i.e., 
\[
\frac{\bigl(1+\|\mu_m(z)\|_X\bigr)^{k-1}}
     {\bigl(1-\|\mu_m(z)\|_X^2\bigr)^k}|\|\mu_m(z)\|^k
\le
\frac{\bigl(1+\|z\|_X^{\,m}\bigr)^{k-1}}
     {\bigl(1-\|z\|_X^{\,2m}\bigr)^k}\|z\|_X^{\,km}\;\text{for all}\;k\geq 1
\]
Now from (\ref{te32}), we deduce that 
\bea\label{te331}|D^kF(\mu_m(z)) ((\mu_m(z))^k)|&\leq&  \frac{k!(1-|F(\mu_m(z)|^2)\bigl(1+|\mu_m(z)|\bigr)^{k-1}}{\bigl(1-|\mu_m(z)|^2\bigr)^k}\nonumber\\
&\leq& \frac{k!(1-|F(\mu_m(z)|^2)\bigl(1+\|z\|_X^{\,m}\bigr)^{k-1}}{\bigl(1-\|z\|_X^{\,2m}\bigr)^k}\|z\|_X^{\,km}.\eea

Assume that $|a|=x\in[0,1]$. If $x=1$, then $P_s(z_0)=0$ for all $s\ge1$, and
\eqref{T2.3} holds trivially. 

Hence, we assume $x\in[0,1)$. Using \eqref{eq:Schwarz} and \eqref{te331}, we obtain
\bea\label{te33} && |F(\mu_m(z))|+\sum_{k=N}^{\infty}\bigg|\frac{D^kF(\mu_m(z)) ((\mu_m(z))^k)}{k!}\bigg|\nonumber\\
&\leq &|F(\mu_m(z))| +\sum_{k=N}^{\infty}\frac{(1-|F(\mu_m(z)|^2)\bigl(1+\|z\|_X^{\,m}\bigr)^{k-1}}{\bigl(1-\|z\|_X^{\,2m}\bigr)^k}\|z\|_X^{\,km}\nonumber\\
&=&|F(\mu_m(z))|+(1-|F(\mu_m(z)|^2)\sum_{k=N}^{\infty}\frac{(1+r^m)^{k-1}}{(1-r^{2m})^k}r^{mk}\nonumber\\
&=&|F(\mu_m(z))|+(1-|F(\mu_m(z)|^2)\sum_{k=N}^{\infty}\frac{r^{mk}}{(1-r^{m})^k(1+r^m)}\nonumber\\
&=&|F(\mu_m(z))|+\frac{r^{mN}}{(1-r^{m})^{N-1}(1+r^m)(1-2r^m)}\Bigl(1-|F(\mu_m(z))|^2\Bigr) \nonumber\\
&\le & \frac{x+r^m}{1+xr^m}+\frac{r^{mN}}{(1-r^{m})^{N-1}(1+r^m)(1-2r^m)}\left(1-\left(\frac{x+r^m}{1+xr^m}\right)^2
\right) \nonumber\nonumber\\
&=&  1+\frac{(1-x)G_3(x,r)}{1+xr^m},
\eea
where $G_3(x,r)=\frac{(1+x)r^{mN}(1-r^{2m})}{(1-r^{m})^{N-1}(1+r^m)(1-2r^m)(1+xr^m)}-(1-r^m).$ The fifth inequality holds for those $r\in[0,1]$ satisfying
$0\le\frac{r^{2mN}}{(1-r^{m})^{N-1}(1+r^m)(1-2r^m)}\le \frac{1}{2},$
that is, for \(r\in[0,r_{3,m,N}]\), where \(r_{3,m,N}\in(0,1)\) is the smallest positive root of
\[
2r^{mN}-(1-r^{m})^{N-1}(1+r^m)(1-2r^m)=0.
\]

Differentiating \(G_3(x,r)\) with respect to \(x\), we obtain
\[
\frac{\partial G_3(x,r)}{\partial x}
=
\frac{(1-r^m)r^{2m}}
{(1-2r^m)(1+xr^m)^2}
\ge 0.
\]
Hence \(G_3(x,r)\) is increasing on \(x\in[0,1]\), and therefore
\[
G_3(x,r)\le G_3(1,r)
=(1-r^m)\left[\frac{2r^{Nm}}{(1-2r^m)(1-r^m)^{N-1}(1+r^m)}-1\right].
\]

Consequently, $G_3(x,r)\le 0$ for \(r\le R_{3,m,N}\), where \(R_{3,m,N}\in(0,1)\) is the smallest positive root of
\begin{equation}\label{eq:R3mk}
2r^{mN}-(1-r^{m})^{N-1}(1+r^m)(1-2r^m)=0.
\end{equation}
We see that $r_{3,m,N}=R_{3,m,N}$.

{\bf \underline{Sharpness of the Radius $R_{3,m,N}$}: }

We now show that the radius $R_{3,m}$ is sharp.  We consider the function $F_2$ defined by (\ref{F2}). Differentiating $k-$ times, we get
\[
 D^kF_2(z)(z^k)
=\frac{k!b^{k-1}(1-b^2)[l_{z_0}(z)]^k}
{\bigl(1-bl_{z_0}(z)\bigr)^{k+1}}.
\]
 Let $\mu_m(z)=l_{z_0}(z)^{m-1}z, m\geq 1.$ Thus for $z=rz_0$, we have 

\begin{align*} & |F(\mu_m(z))|+\sum_{k=N}^{\infty}\bigg|\frac{D^kF(\mu_m(z)) ((\mu_m(z))^k)|}{k!}\bigg|\nonumber\\
=&\frac{b-r^m}{1-br^m}+\frac{(1-b^2)}{b(1-br^m)}\sum_{n=N}^{\infty}\left(\frac{br^m}{1-br^m}\right)^n\nonumber\\
&=1+\frac{(1-b)D_3(b,r)}{1-br^m},
\end{align*}
where $D_3(b,r)=-(1+r^m)+\frac{b^{N-1}(1+b)r^{mN}}{(1-br^m)^{N-1}(1-2br^m)}.$

It is evident that
\[
\lim_{b\to 1^-}D_3(b,r)
=-(1+r^m)+\frac{2r^{mN}}{(1-r^m)^{N-1}(1-2r^m)}>0,
\]
for \(r>R_{3,m}\), where \(R_{3,m}\) is the smallest positive root of the equation
\begin{equation}\label{eq:R3mk}
2r^{mN}-(1-r^m)^{N-1}(1+r^{m})(1-2r^m)=0.
\end{equation}

\end{proof}
\begin{rem} Theorem \ref{T3} yields several interesting special cases, which are discussed below together with a number of useful observations.
\begin{enumerate} 
\item[(i)] Theorem \ref{T3} reduces to Theorem A in \cite{LSX2018} when $B_X=\mathbb{U}$, $m=1$ and $\mu_1(z)=z$.
\item[(i)] Theorem~\ref{T3} recovers Theorem~3.4 of \cite{ABM2026} as the case $N=2$, $B_X=\mathbb{U}$, $m=1$, and $\mu_1(z)=z$. Thus, our result may be viewed as an extension of \cite[Theorem 3.4]{ABM2026} from the second-order case to arbitrary order $N\ge2$.
\end{enumerate} 
\end{rem}
In table \ref{tab3.3}, Figure \ref{f35}, we obtain the values of $R_{3,m,N}$ for certain values of $m,N\in\mathbb{N}$.

\begin{table}[H]
\centering
\begin{minipage}[b]{0.45\textwidth}
\centering
\begin{tabular}{|c|c|c|}
\hline
\textbf{m} & \textbf{N} &\textbf{$R_{3,m,N}$}  \\ 
\hline
1 & 2 & 0.355420\\ 
\hline
2 & 3 & 0.622760 \\ 
\hline
3 & 4 & 0.740940 \\ 
\hline
4 & 5 & 0.804770 \\ 
\hline
\end{tabular}
\caption{$R_{3,m,N}$ is the smallest root of equation (\ref{eq:R3mk}) in $(0,1)$.}
\label{tab3.3}
\end{minipage}%
\hfill
\begin{minipage}[b]{0.55\textwidth}
\centering
\begin{figure}[H]
\centering
\includegraphics[width=\textwidth]{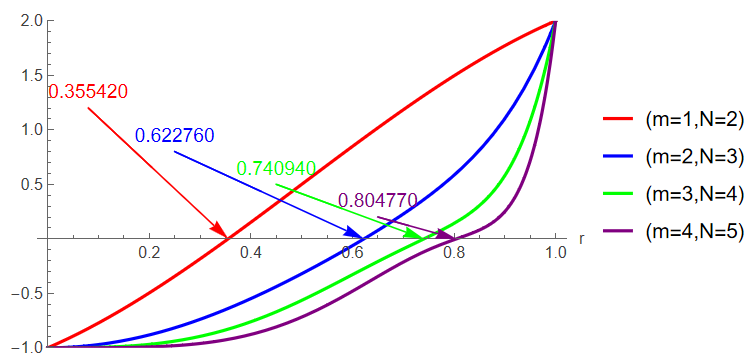}
\caption{The graphs exhibit the locations of the roots $R_{3,m,N}$ in $(0,1)$ for different value of $m,N$.}
\label{f35}
\end{figure}
\end{minipage}

\end{table}
\begin{theo}\label{T4}
Let $B_{X}$ be the unit ball of a complex Banach space $X$ also let $F$ be a holomorphic mapping from $B_{X}$ to $\ol{\mathbb{U}}$ with
\[
F(z)=a+\sum_{s=1}^{\infty}P_{s}(z), \qquad z\in B_{X},
\]
where $P_{s}(z)=\frac{1}{s!}D^{s}F(0)(z^{s})$ and   $\mu_m:B_{X}\to B_{X}$ is  Schwarz mappings having $z=0$ as a zero of order $m$. Then, for each $N,m\in\mathbb{N}$, the inequality
\bea\label{T2.4}&&|F(\mu_m(z))|^2+\sum_{k=N}^{\infty}\bigg|\frac{D^kF(\mu_m(z)) ((\mu_m(z))^k)}{k!}\bigg|\leq 1\
\eea
for $\|z\|=r\leq R_{4,m,N}$, where  $R_{4,m,N}\in(0,1)$ is the smallest positive root of the equation
\[
r^{mN}-(1-r^{m})^{N-1}(1+r^m)(1-2r^m)=0.
\]
The number $R_{4,m,N}$ cannot be improved.
\end{theo}

\begin{proof}[{\bf Proof of the Theorem \ref{T4} }]
Using similar argument as in the proof of Theorems \ref{T1} and \ref{T2}, and in view of Lemma \ref{L20}, we obtain (\ref{te33}). Therefore
\bea\label{te33} && |F(\mu_m(z))|^2+\sum_{k=N}^{\infty}\bigg|\frac{D^kF(\mu_m(z)) ((\mu_m(z))^k)}{k!}\bigg|\nonumber\\
&\leq &|F(\mu_m(z))|^2 +\sum_{k=N}^{\infty}\frac{(1-|F(\mu_m(z)|^2)\bigl(1+\|z\|_X^{\,m}\bigr)^{k-1}}{\bigl(1-\|z\|_X^{\,2m}\bigr)^k}\|z\|_X^{\,km}\nonumber\\
&=&|F(\mu_m(z))|^2+(1-|F(\mu_m(z)|^2)\sum_{k=N}^{\infty}\frac{(1+r^m)^{k-1}}{(1-r^{2m})^k}r^{mk}\nonumber\\
&=&|F(\mu_m(z))|^2+(1-|F(\mu_m(z)|^2)\sum_{k=N}^{\infty}\frac{r^{mk}}{(1-r^{m})^k(1+r^m)}\nonumber\\
&=&|F(\mu_m(z))|^2+\frac{r^{mN}}{(1-r^{m})^{N-1}(1+r^m)(1-2r^m)}\Bigl(1-|F(\mu_m(z))|^2\Bigr) \nonumber\\
&\le & \left(\frac{x+r^m}{1+xr^m}\right)^2+\frac{r^{mN}}{(1-r^{m})^{N-1}(1+r^m)(1-2r^m)}\left(1-\left(\frac{x+r^m}{1+xr^m}\right)^2
\right) \nonumber\nonumber\\
&=&  1+\frac{(1-x^2)(1-r^{2m})G_4(x,r)}{(1+xr^m)^2},
\eea
where $G_4(x,r)=\frac{r^{mN}}{(1-r^{m})^{N-1}(1+r^m)(1-2r^m)}-1$. The fifth inequality holds for those $r\in[0,1]$ satisfying
$0\le\frac{r^{mN}}{(1-r^{m})^{N-1}(1+r^m)(1-2r^m)}\le 1,$
that is, for \(r\in[0,r_{4,m,N}]\), where \(r_{4,m,N}\in(0,1)\) is the smallest positive root of
\[
r^{mN}-(1-r^{m})^{N-1}(1+r^m)(1-2r^m)=0,
\]

Differentiating \(G_4(x,r)\) with respect to \(x\), we obtain
\[
\frac{\partial G_4(x,r)}{\partial x}
= 0.
\]
Hence \(G_4(x,r)\) is constant on \(x\in[0,1]\), and therefore
\[
G_4(x,r)= G_4(1,r)
=
\frac{r^{mN}}{(1-r^{m})^{N-1}(1+r^m)(1-2r^m)}-1.
\]

Consequently, $G_4(x,r)\le 0$ for \(r\le R_{4,m,N}\), where \(R_{4,m,N}\in(0,1)\) is the smallest positive root of
\begin{equation}\label{eq:R4mk}
r^{mN}-(1-r^{m})^{N-1}(1+r^m)(1-2r^m)=0.
\end{equation}
We see that $r_{4,m,N}=R_{4,m,N}$.

{\bf \underline{Sharpness of the Radius $R_{4,m}$}: }

We now show that the radius $R_{4,m}$ is sharp.  We consider the function $F_2$ defined by (\ref{F2}). Differentiating $k-$ times, we get
\[
 D^kF_2(z)(z^k)
=\frac{k!b^{k-1}(1-b^2)[l_{z_0}(z)]^k}
{\bigl(1-bl_{z_0}(z)\bigr)^{k+1}}.
\]
 Let $\mu_m(z)=l_{z_0}(z)^{m-1}z, m\geq 1.$ Thus for $z=rz_0$, we have 

\begin{align*} & |F(\mu_m(z))|^2+\sum_{k=N}^{\infty}\bigg|\frac{D^kF(\mu_m(z)) ((\mu_m(z))^k)|}{k!}\bigg|\nonumber\\
=&\left(\frac{b-r^m}{1-br^m}\right)^2+\frac{(1-b^2)}{b(1-br^m)}\sum_{n=N}^{\infty}\left(\frac{br^m}{1-br^m}\right)^n\nonumber\\
=&\left(\frac{b-r^m}{1-br^m}\right)^2+\frac{(1-b^2)b^{N-1}r^{mN}}{(1-2br^m)(1-br^m)^N}\nonumber\\
&=1+\frac{(1-b^2)D_4(b,r)}{(1-br^m)^2},
\end{align*}
where $D_4(b,r)=-(1-r^{2m})+\frac{b^{N-1}r^{mN}}{(1-2br^m)(1-br^m)^{N-2}}.$

It is evident that
\[
\lim_{b\to 1^-}D_4(b,r)
=-(1-r^{2m})+\frac{r^{mN}}{(1-2r^m)(1-r^m)^{N-2}}=\frac{r^{mN}-(1-2r^m)(1-r^m)^{N-1}(1+r^m)}{(1-2r^m)(1-r^m)^{N-2}}>0,
\]
for \(r>R_{4,m,N}\), where \(R_{4,m,N}\) is the smallest positive root of the equation
\[
r^{mN}-(1-r^{m})^{N-1}(1+r^m)(1-2r^m)=0.
\]

\end{proof}
\begin{rem} Theorem \ref{T4} yields several interesting special cases, which are discussed below together with a number of useful observations.
\begin{enumerate}
\item[(i)] Theorem \ref{T4} reduces to Theorem B in \cite{LSX2018} when $B_X=\mathbb{U}$, $m=1$ and $\mu_1(z)=z$.
\item[(ii)]  It is worth noting that Theorem~\ref{T4} extends Theorem~3.6 of \cite{ABM2026} in the direction of higher-order derivative Bohr inequalities. Indeed, when $N=2$, $B_X=\mathbb{U}$, $m=1$, and $\mu_1(z)=z$, Theorem~\ref{T4} reduces to Theorem~3.6 of \cite{ABM2026}. However, Theorem~\ref{T4} remains valid for every integer $N\ge2$, thereby providing a family of sharp Bohr-type inequalities of arbitrary order.
\end{enumerate}\end{rem} 
In table \ref{tab3.4}, Figure \ref{f41}, we obtain the values of $R_{4,m,N}$ for certain values of $m,N\in\mathbb{N}$.

\begin{table}[H]
\centering
\begin{minipage}[b]{0.45\textwidth}
\centering
\begin{tabular}{|c|c|c|}
\hline
\textbf{m} & \textbf{N} &\textbf{$R_{4,m,N}$}  \\ 
\hline
1 & 2 & 0.403030 \\ 
\hline
2 & 3 & 0.649140 \\ 
\hline
3 & 4 & 0.756450\\ 
\hline
4 & 5 & 0.814850 \\ 
\hline
\end{tabular}
\caption{$R_{4,m,N}$ is the smallest root of equation (\ref{eq:R4mk}) in $(0,1)$.}
\label{tab3.4}
\end{minipage}%
\hfill
\begin{minipage}[b]{0.55\textwidth}
\centering
\begin{figure}[H]
\centering
\includegraphics[width=\textwidth]{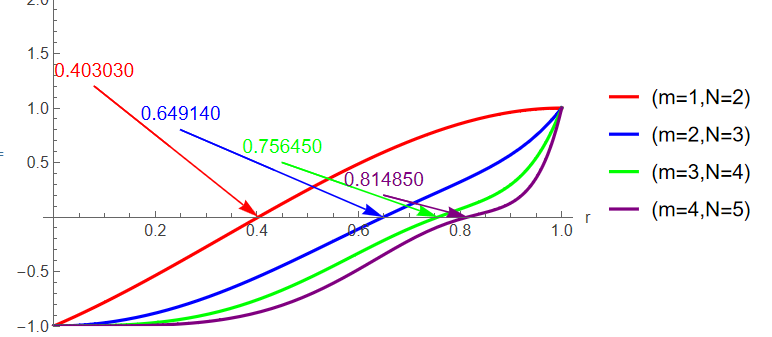}
\caption{The graphs exhibit the locations of the roots $R_{4,m,N}$ in $(0,1)$ for different value of $m,N$.}
\label{f41}
\end{figure}
\end{minipage}

\end{table}

\section*{Declarations}

\subsection*{Funding}
Not applicable.

\subsection*{Data Availability Statement}
Not applicable.

\subsection*{Conflict of Interest}
The authors declare that there is no conflict of interest.

\subsection*{Clinical Trial Number} not applicable.
\subsection*{Author Contribution} All authors contributed equally to this work

\end{document}